\documentclass[a4paper,12pt]{article}

\usepackage{amsmath}
\usepackage{amsthm}
\usepackage{amssymb}

\usepackage[english]{babel}

\usepackage{tikz}

\usepackage{comment}
\usepackage{enumitem}
\usepackage{authblk}

\theoremstyle{plain}
\newtheorem{thm}{Theorem}[section]

\newtheorem{prop}[thm]{Proposition}

\newtheorem{con}[thm]{Conjecture}

\theoremstyle{remark}

\newtheorem*{rem*}{Remark}

\newcommand{\N}{{\mathbb{N}}} 

\DeclareSymbolFont{bbold}{U}{bbold}{m}{n}
\DeclareSymbolFontAlphabet{\mathbbold}{bbold}

\numberwithin{equation}{section}

\newcommand{\Addresses}{{
  \bigskip
  \footnotesize

	\noindent
  Sigrid Grepstad, \textsc{Department of Mathematical Sciences, Norwegian University of 
	Science and Technology, 7491 Trondheim, Norway.} 
	\par\nopagebreak
	\noindent \textit{E-mail address: }\texttt{sigrid.grepstad@ntnu.no}

  \medskip
	\noindent
  Lisa Kaltenb\"ock and Mario Neum\"uller, \textsc{Department of Financial Mathematics and applied 
	Number Theory, Johannes Kepler University, Altenbergerstra{\ss}e 69, 4040 
	Linz, Austria}
  \par\nopagebreak
  \noindent \textit{E-mail addresses: }\texttt{lisa.kaltenboeck@jku.at} and 
	\texttt{mario.neumueller@jku.at}}}

\allowdisplaybreaks

\title{On the asymptotic behaviour of the sine product $\prod_{r=1}^n |2\sin \pi r \alpha|$}

\author{Sigrid Grepstad, Lisa Kaltenb\"ock and Mario Neum\"uller \thanks{Sigrid Grepstad is supported in part by Grant 275113 of the Research Council of Norway.
Lisa Kaltenb\"{o}ck and Mario Neum\"{u}ller are funded by the Austrian Science Fund (FWF): Project F5507-N26 and Project F5509-N26, which are both part of the Special Research 
Program ``Quasi-Monte Carlo Methods: Theory and Applications''. 
}}

\date{}

\begin{document}


	\maketitle
	
  \begin{abstract}
	In this paper we review recently established results on the asymptotic behaviour	of the trigonometric product $P_n(\alpha) = \prod_{r=1}^n |2\sin \pi r \alpha|$ as $n\to \infty$. We focus on irrationals $\alpha$ whose continued fraction coefficients are bounded. Our main goal is to illustrate that when discussing the regularity of $P_n(\alpha)$, not only the boundedness of the coefficients plays a role; also their size, as well as the structure of the continued fraction expansion of $\alpha$, is important. 
	\end{abstract}
	
  \centerline{\begin{minipage}[hc]{130mm}{
	}
	\end{minipage}}

\section{Introduction}\label{sec:intro}
The trigonometric product 
\begin{equation*}
P_n(\alpha) = \prod_{r=1}^n |2\sin \pi r \alpha|
\end{equation*}
has been subject to mathematical investigations for more than 50 years. It arises naturally in a number of mathematical fields, such as partition theory, Pad\'{e} approximation and discrepancy theory. Of particular interest is the asymptotic behaviour of $P_n(\alpha)$ as $n \to \infty$, which has proven surprisingly difficult to determine. In Figure \ref{fig:spex}, we have plotted $P_n(\alpha)$ for $n=1, \ldots , 250$ and different values of irrational $\alpha$. These plots illustrate the chaotic nature of the product sequence $P_n(\alpha)$. Yet we see that for certain values of $\alpha$, there is some self-similarity in the behaviour of $P_n(\alpha)$ with increasing $n$.
\begin{figure}[htb]
	\includegraphics[width=0.49\linewidth]{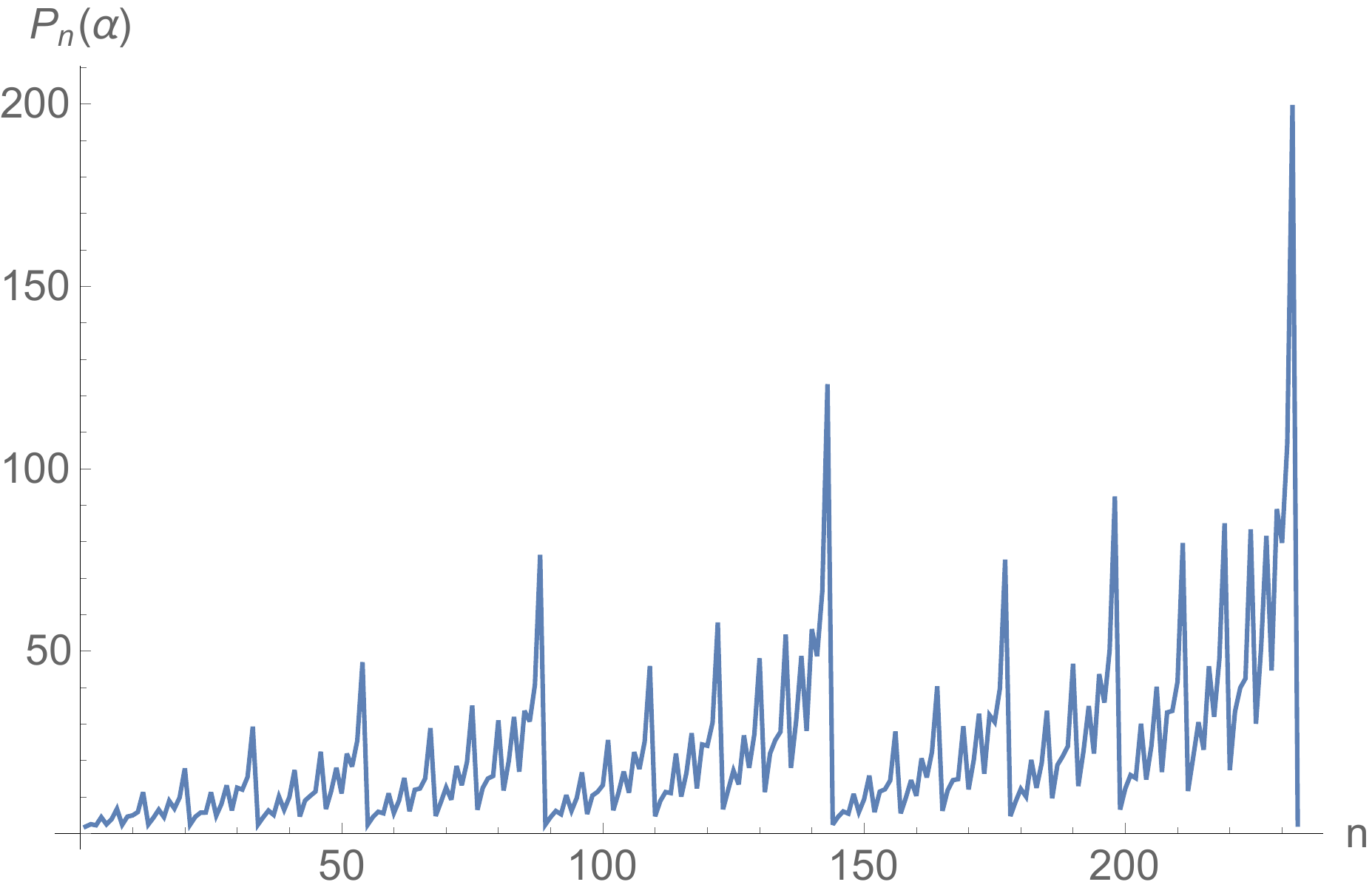}
	\includegraphics[width=0.49\linewidth]{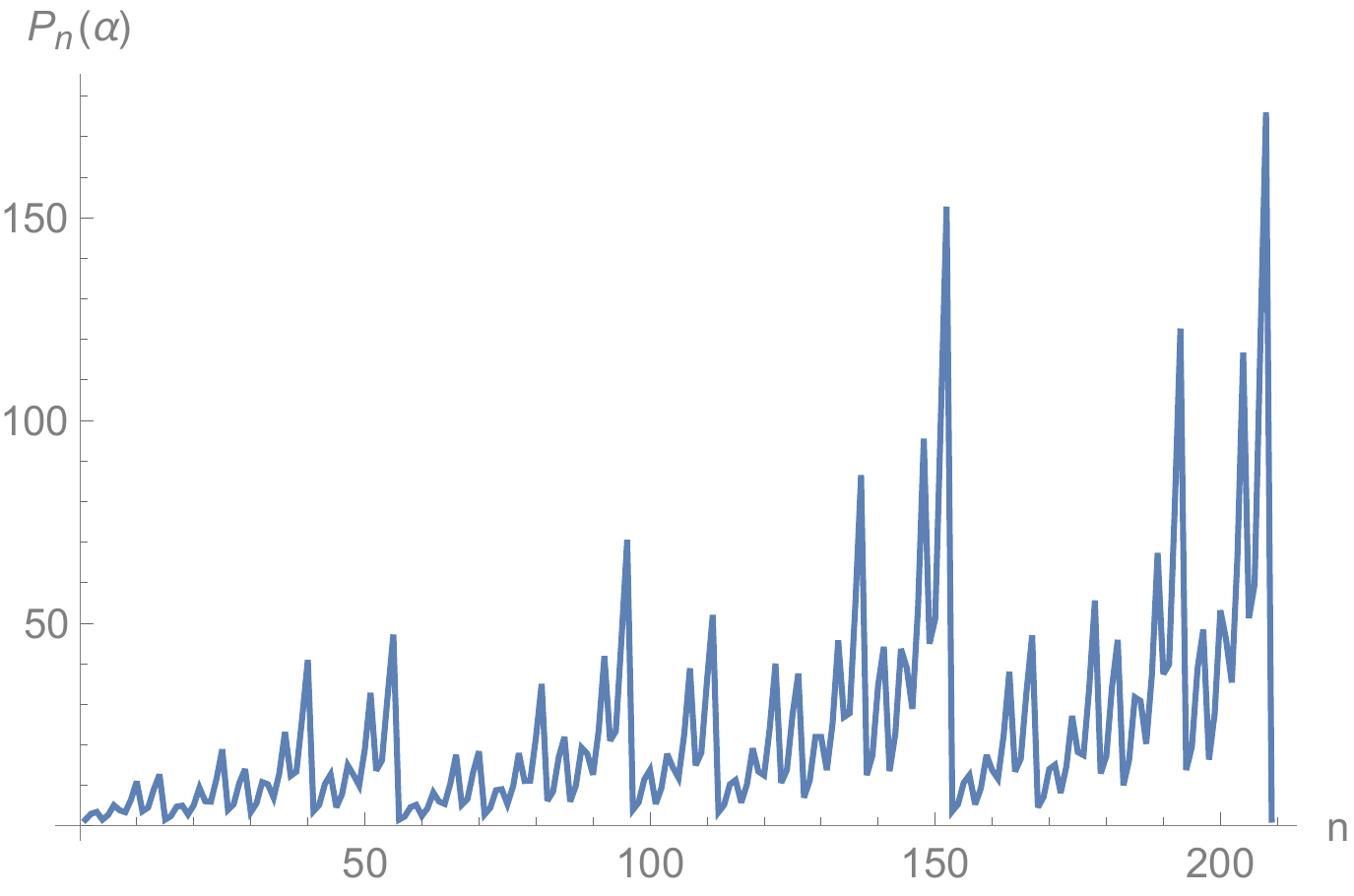}
	\includegraphics[width=0.49\linewidth]{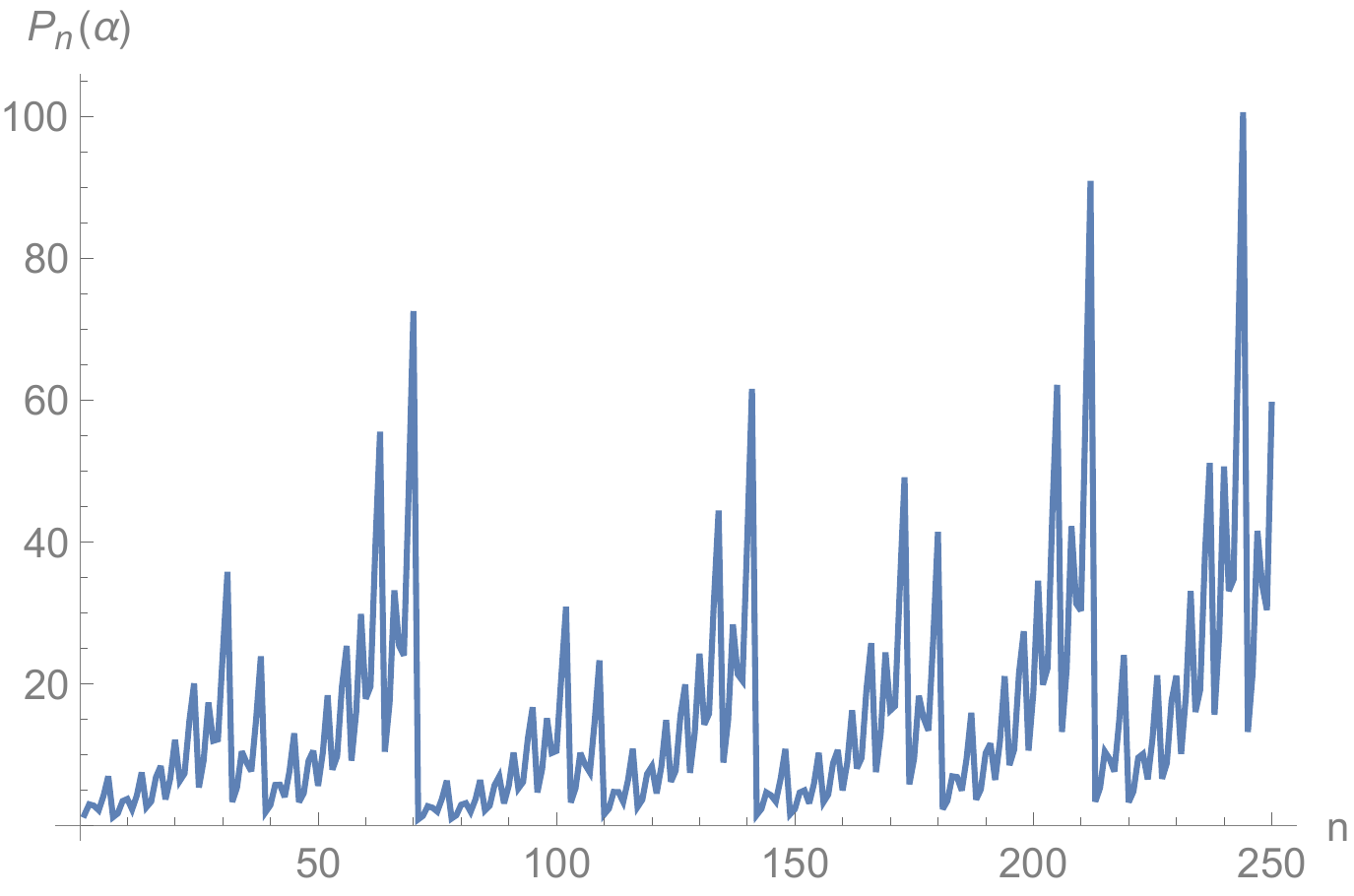}
	\includegraphics[width=0.49\linewidth]{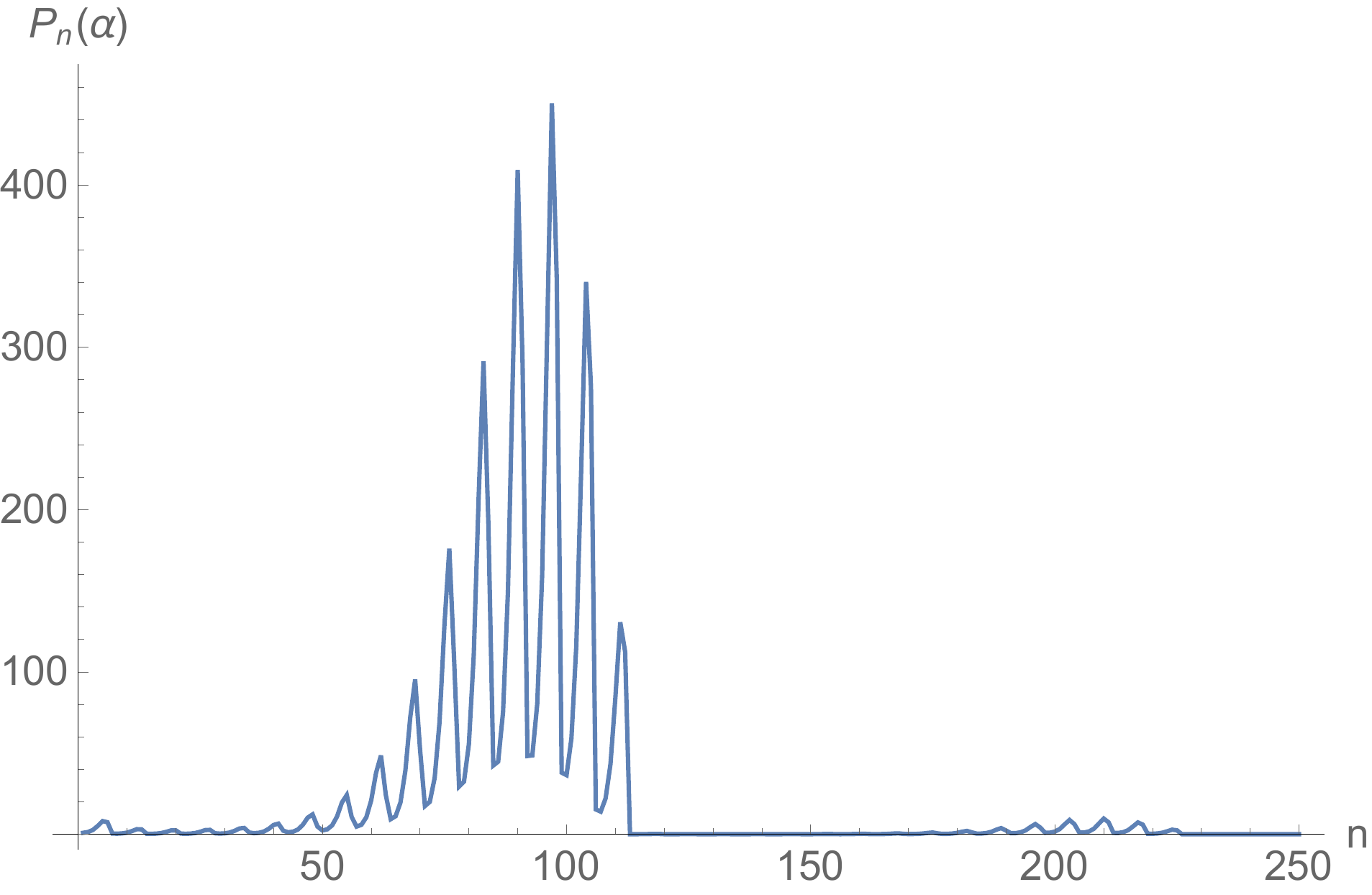}
	\caption{Values of $P_n(\alpha)$ for $\alpha=(\sqrt{5}-1)/2$ (upper left), $\alpha = \sqrt{3}$ (upper right), $\alpha = e$ (lower left) and $\alpha=\pi$ (lower right).}
	\label{fig:spex}
\end{figure}

In this paper we review known bounds on the growth and decay of $P_n(\alpha)$, focusing on breakthroughs in the last 5 years. These recent developments deal mainly with the case when $\alpha$ has bounded continued fraction coefficients. As shown by Lubinsky 20 years ago, this is a case in which the behaviour of $P_n(\alpha)$ is exceptionally regular (see Section~\ref{subsec:cfc}). What recent results have come to reveal, is that also the structure of the continued fraction expansion of $\alpha$ affects regularity. For instance, certain limit phenomena appear only for very structured expansions (see Section \ref{sec:subseqs}). Moreover, and perhaps more surprisingly, also the specific sizes of the continued fraction coefficients play a role. This is evident when discussing the long-standing open question (now resolved) of whether $\liminf_{n\to \infty} P_n(\alpha)=0$ for all irrationals $\alpha$.

\subsection{Growth of $P_n(\alpha)$}
Let us briefly review what is known about the growth of $P_n(\alpha)$ as $n\to \infty$. Note first that if $\alpha=p/q$ is rational, then $P_n(\alpha)=0$ for all $n\geq q$. Moreover, we have that $P_n(\alpha) = P_n(\{\alpha\})$, where $\{\cdot\}$ denotes the fractional part, so we may safely restrict our attention to irrationals $\alpha$ in the unit interval. 

It was established by Sudler \cite{Su64} and Wright \cite{Wr64} in the 1960s that the norm $\| P_n(\alpha)\| = \sup_{0 < \alpha < 1} |P_n(\alpha)|$ grows exponentially as $n\to \infty$, and 
\begin{equation}
\label{eq:expgrowth}
\lim_{n\to \infty} \| P_n(\alpha) \|^{1/n} = C \approx 1.22. 
\end{equation}
(See also \cite{freiman} for an alternative approach and the exact value of $C$.) In light of \eqref{eq:expgrowth}, one might expect that also the pointwise growth of $P_n(\alpha)$ is exponential, but this is not the case. It was shown by Lubinsky and Saff in \cite{LuSaff} that for almost every $\alpha \in (0,1)$, we have 
\begin{equation*}
\lim_{n\to \infty} P_n(\alpha)^{1/n} = 1.
\end{equation*}
In later work, Lubinsky provides a more precise growth bound on $P_n(\alpha)$, namely 
\begin{equation}
\label{eq:uppermax}
\left| \log P_n(\alpha) \right| = O \left( \log n (\log \log n)^{1+\varepsilon} \right)
\end{equation}
for any $\varepsilon>0$, and this holds for almost every $\alpha$ \cite{Lu99}. In the opposite direction, $P_n(\alpha)$ grows almost linearly for infinitely many $n$. We have that 
\begin{equation*}
\limsup_{n\to \infty} \frac{\log P_n(\alpha)}{\log n} \geq 1 
\end{equation*}
for all irrationals $\alpha \in (0,1)$.

\subsection{Significance of the continued fraction expansion}\label{subsec:cfc}
In his 1999 paper \cite{Lu99}, Lubinsky illustrates a significant difference in nature of $P_n(\alpha)$ depending on whether or not the continued fraction expansion of $\alpha$ has bounded coefficients. If this is the case, then there exist positive constants $C_1$ and $C_2$ such that
\begin{equation}
\label{eq:polybounds}
n^{-C_2} \leq P_n(\alpha) \leq n^{C_1},
\end{equation}
i.e.\ $P_n(\alpha)$ can be polynomially bounded (see \cite[Theorem 1.3]{Lu99}). 

When $\alpha$ has \emph{unbounded} continued fraction coefficients, the upper bound in \eqref{eq:uppermax} (valid for almost all such $\alpha$) has yet to be improved upon. Moreover, Lubinsky showed that 
\begin{equation}
\label{eq:unboundedzero}
\liminf_{n\to \infty} P_n(\alpha) = 0
\end{equation}
in this case, and that for almost all $\alpha$ the decay to $0$ is faster than any negative power of $n$ for infinitely many $n$. 

The focus of this paper will be on the more regular case when $\alpha$ has bounded continued fraction coefficients, and on two closely related questions raised by Lubinsky in \cite{Lu99}, namely:\\

\begin{enumerate}
	\item Does \eqref{eq:unboundedzero} still hold in the case of bounded continued fraction coefficients? 
	\item What is the smallest value we can choose for $C_2$ in \eqref{eq:polybounds}?
\end{enumerate}

 Our interest in these questions was sparked by a recent paper by Mestel and Verschueren \cite{MV16}, where the special case $\alpha = (\sqrt{5}-1)/2$ is studied in great detail. We review key results from this paper in Section \ref{sec:subseqs}. Using these key results, we argue in Section \ref{sec:liminf} that for $\alpha=(\sqrt{5}-1)/2$, equality \eqref{eq:unboundedzero} does \emph{not} hold. We will see in Section \ref{sec:extensions} that, in fact, it appears one may choose $C_2=0$ for this specific $\alpha$. In the same section we explain why this simplest choice of $C_2$ cannot possibly be valid for all $\alpha$ with bounded continued fraction coefficients; this was also alluded to by Lubinsky in \cite{Lu99}.

A third question natural to raise is: what is the smallest value we may choose for $C_1$ in \eqref{eq:polybounds}? We firmly believe that for the special case $\alpha = (\sqrt{5}-1)/2$, the answer to this question is $C_1=1$ (see Figure \ref{fig:linear}). More precisely, we believe that $P_n(\alpha) < cn$ for some constant $c>0$ independent of $n$. Upper bounds on $P_n(\alpha)$ will not be the focus of this paper. Nevertheless, we will briefly return to this question for the special case $\alpha=(\sqrt{5}-1)/2$ in Section \ref{sec:liminf}.
\begin{figure}[htb]
	\centering
		\includegraphics[width=0.65\linewidth]{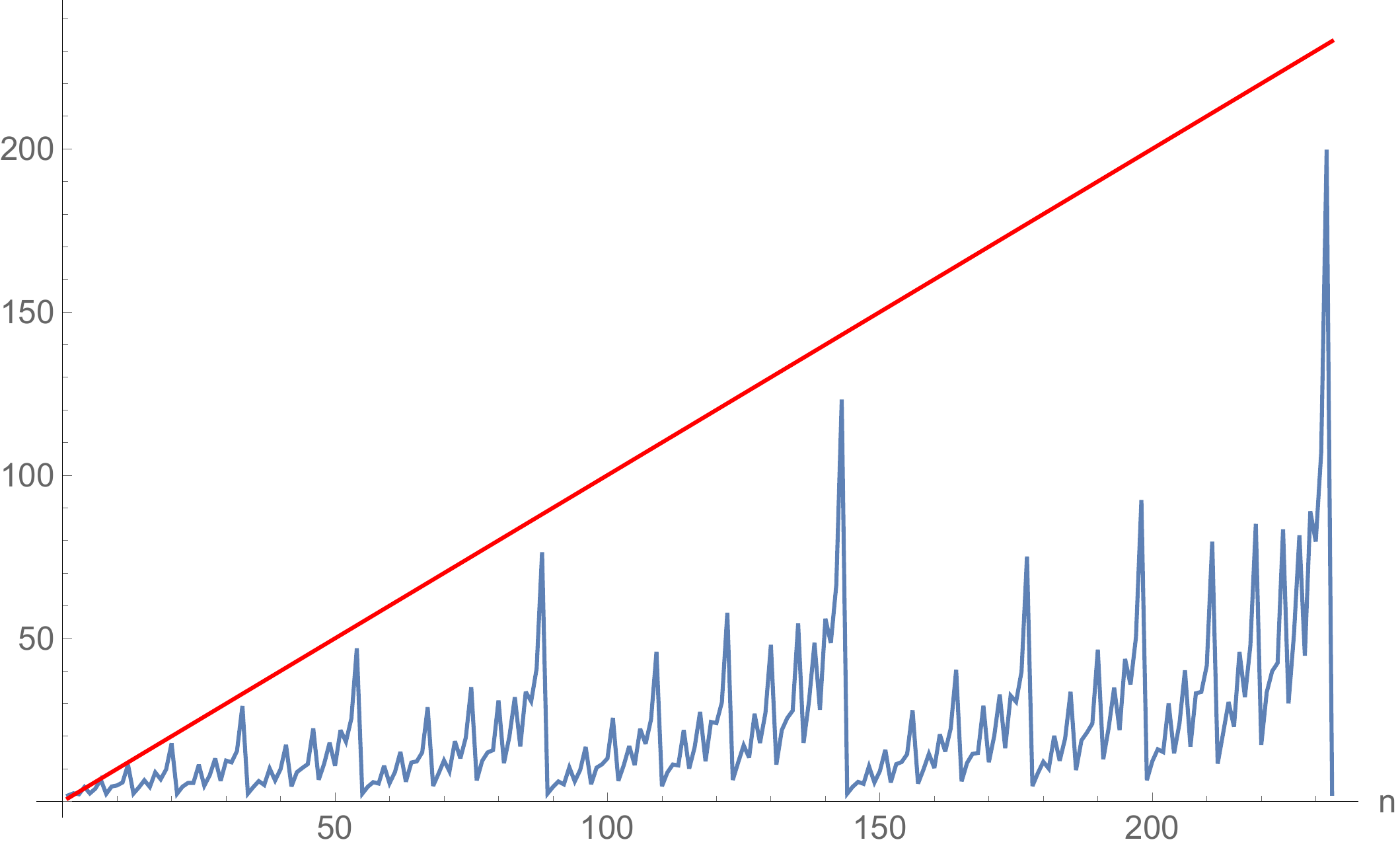}
	\caption{Value of $P_n(\alpha)$ for $\alpha = (\sqrt{5}-1)/2$ (blue line) plotted against $f(n)=n$ (red line).\label{fig:linear}}
\end{figure}


\section{Continued fraction expansions}\label{sec:cfe}
In order to set the notation for the remainder of the paper,
we briefly review some facts about continued fraction expansions. Any irrational $\alpha \in (0,1)$ has a unique and infinite continued fraction expansion 
$$\alpha= \frac{1}{\displaystyle a_1+\frac{1}{\displaystyle a_2+\frac{1}{\displaystyle a_3+ \ldots}}} = [0;a_1, a_2, a_3, \ldots ],$$
where $a_i \in \N$ for all $i \in \N$. A best rational approximation of $\alpha$ is given by $p_n/q_n$, where $p_n$ and $q_n$ are defined recursively by
\begin{align*}
q_0 &=0, \quad q_1=1, \quad  q_{n+1}=a_nq_n+ q_{n-1}; \\
p_0&=1, \quad p_1=0, \quad p_{n+1}=a_np_n + p_{n-1}.
\end{align*}
This approximation is best possible in the sense that for no $q < q_n$ can we find $p \in \N$ such that 
\begin{equation*}
\left| \alpha - \frac{p}{q}\right| < \left| \alpha - \frac{p_n}{q_n} \right| .
\end{equation*}
We call $p_n$ and $q_n$ the best approximation numerator and denominator of $\alpha$, respectively. The fraction $p_n/q_n$ is called the $n$th convergent of $\alpha$, and it is well-known that
\begin{equation}
\label{eq:approxerror}
\left| \alpha - \frac{p_n}{q_n}\right| \leq \frac{1}{q_{n+1}q_n}  .
\end{equation}
Finally, we recall that given a sequence of best approximation denominators $\{q_0, q_1, q_2, \ldots \}$ corresponding to some irrational $\alpha$, any natural number $N$ has a unique \emph{Ostrowski expansion} in terms of this sequence.
\begin{thm}[Ostrowski representation]\label{thm:ostrowski}
Let $\alpha \in (0,1)$ be an irrational with continued fraction expansion $[0;a_1, a_2, \ldots]$ and best approximation denominators $(q_n)_{n\geq 1}$. Then every natural number $N$ has a unique expansion 
\begin{equation}
\label{eq:ostrowski}
N = \sum_{j=1}^z b_j q_j ,
\end{equation}
where 
\begin{enumerate}
\setlength\itemsep{5pt}
\item $0 \leq b_1 \leq a_1-1$ and $0\leq b_j \leq a_j$ for $j>1$.
\item If $b_j = a_j$ for some $j$, then $b_{j-1}=0$.
\item $z = z(N) = O(\log N)$.
\end{enumerate}
We refer to \eqref{eq:ostrowski} as the Ostrowski representation of $N$ in base $\alpha$.
\end{thm}
A proof of Theorem \ref{thm:ostrowski} can be found in \cite[p.~126]{KN74}. For further reading on the Ostrowski expansion, see \cite{AS03} or \cite{RS92}.


\section{Convergence along subsequences}\label{sec:subseqs}
In a recent paper by Mestel and Verschueren \cite{MV16}, the authors give a detailed exposition on the product $P_n(\alpha)$ in the special case when $\alpha = \varphi=(\sqrt{5}-1)/2$ is the (fractional part of the) golden mean. The irrational number $\varphi$ has the simplest possible continued fraction expansion 
$$\varphi= \frac{1}{\displaystyle 1+\frac{1}{\displaystyle 1+\frac{1}{\displaystyle 1+ \ldots}}} = [0;\overline{1}],$$
and the sequence of best approximation denominators of $\varphi$ is the well-known Fibonacci sequence
\begin{equation}
\label{eq:fibonacci}
(F_n)_{n\geq 0} = (0,1,1,2,3,5,8,13, \ldots).
\end{equation}
Mestel and Verschueren give a rigorous proof of an intriguing fact which was observed experimentally in \cite{KnTa2011} by Knill and Tangerman, namely that the subsequence $P_{F_n}(\varphi)$ converges to a positive constant as $n \to \infty$.
\begin{thm}[{\cite[Theorem 3.1]{MV16}}]
\label{thm:mvmain}
Let $\varphi=(\sqrt{5}-1)/2$ and let $(F_n)_{n\geq 0}$ be the Fibonacci sequence in \eqref{eq:fibonacci}. The subsequence $(P_{F_n}(\varphi))_{n\geq 1}$ is convergent, and 
\begin{equation*}
\lim_{n\to \infty} P_{F_n}(\varphi) = \lim_{n\to \infty} \prod_{r=1}^{F_n} \left| 2 \sin \pi r \varphi \right| > 0. 
\end{equation*}
\end{thm}
Numerical calculations suggest that the limiting value of $P_{F_n}(\varphi)$ is approximately $2.4$ (see Figure \ref{fig:convergenceex}).

It turns out that the convergence of the subsequence $P_{F_n}(\varphi)$ is not a property specific to the golden mean. The same property can be established for any irrational $\alpha$ with continued fraction expansion $\alpha = [0;\overline{a}] $, and a similar phenomenon is observed for any irrational with a periodic continued fraction expansion.
\begin{thm}[{\cite[Theorem 1.2]{GrNe18}}]
\label{thm:periodcase}
Suppose $\alpha$ has a periodic continued fraction expansion of the form $\alpha=[0;\overline{a_1,a_2, \ldots, a_{\ell}}]$ with period $\ell$, and let $(q_{n})_{n\geq 0}$ be its sequence of best approximation denominators. Then there exist positive constants $C_0, C_1, \ldots, C_{\ell-1}$ such that 
\begin{equation*}
	\lim_{m\to\infty}P_{q_{\ell m+k}}(\alpha)=\lim_{m \to \infty} \prod_{r=1}^{q_{\ell m+k}} \left| 2 \sin \pi r \alpha \right| = C_k
\end{equation*}
for each $k=0,1, \ldots, \ell-1$.
\end{thm}
Adding a preperiod to the continued fraction expansion of $\alpha$ in Theorem~\ref{thm:periodcase} does not alter the conclusion, and accordingly this result extends to all quadratic irrationals $\alpha$. See \cite{GrNe18} for further details. 

In Figure \ref{fig:convergenceex} below, we have plotted the subsequences $P_{q_n}(\alpha)$ for $\alpha = \varphi$ and $\alpha=\sqrt{3}$. In the latter case, the continued fraction expansion of $\alpha$ has period $\ell=2$, and accordingly we observe that the two subsequences $P_{q_{2m}}(\alpha)$ and $P_{q_{2m+1}}(\alpha)$ converge rapidly to two different positive constants. 

\begin{figure}[htb]
\centering
	\includegraphics[width=0.65\linewidth]{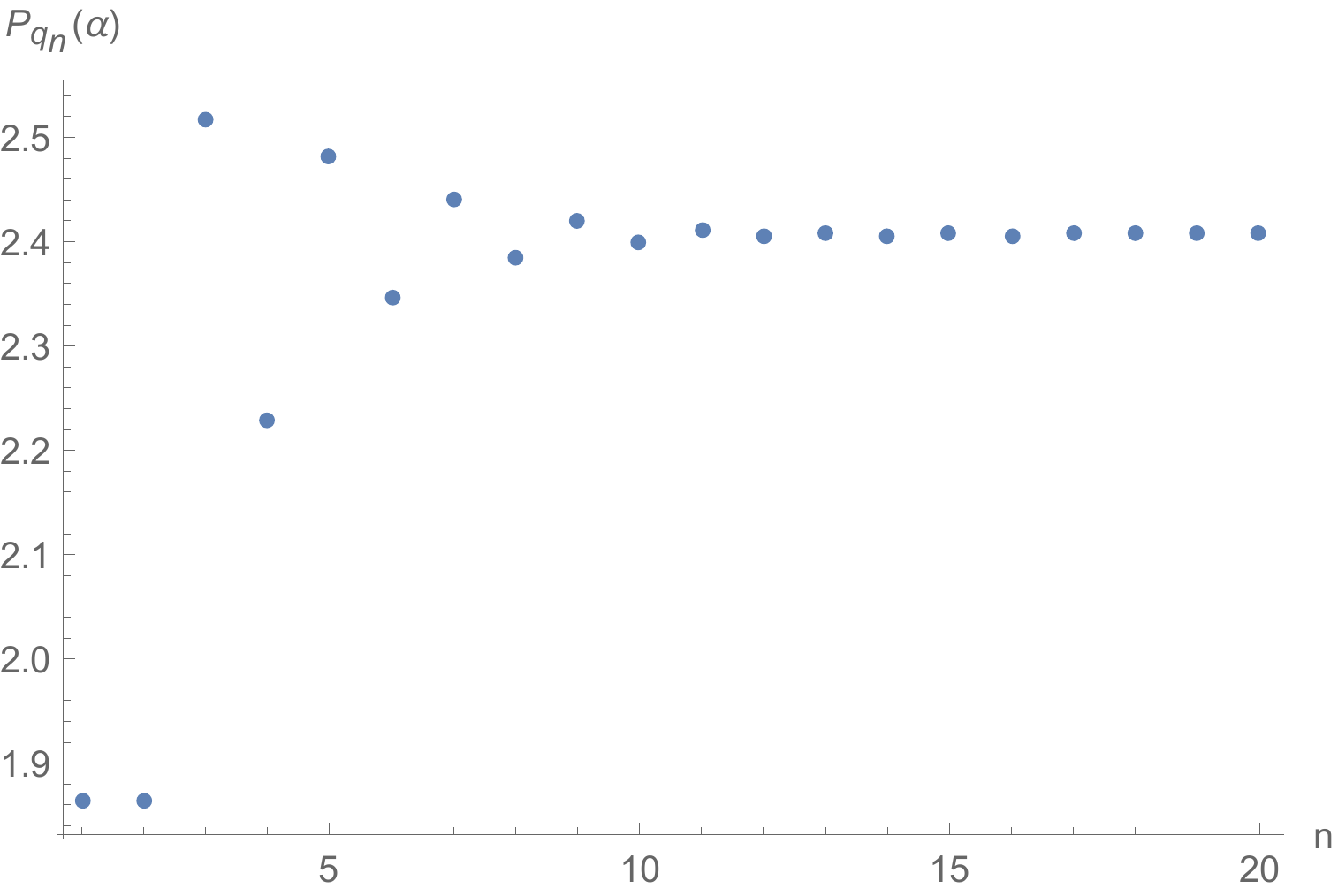}
	\includegraphics[width=0.65\linewidth]{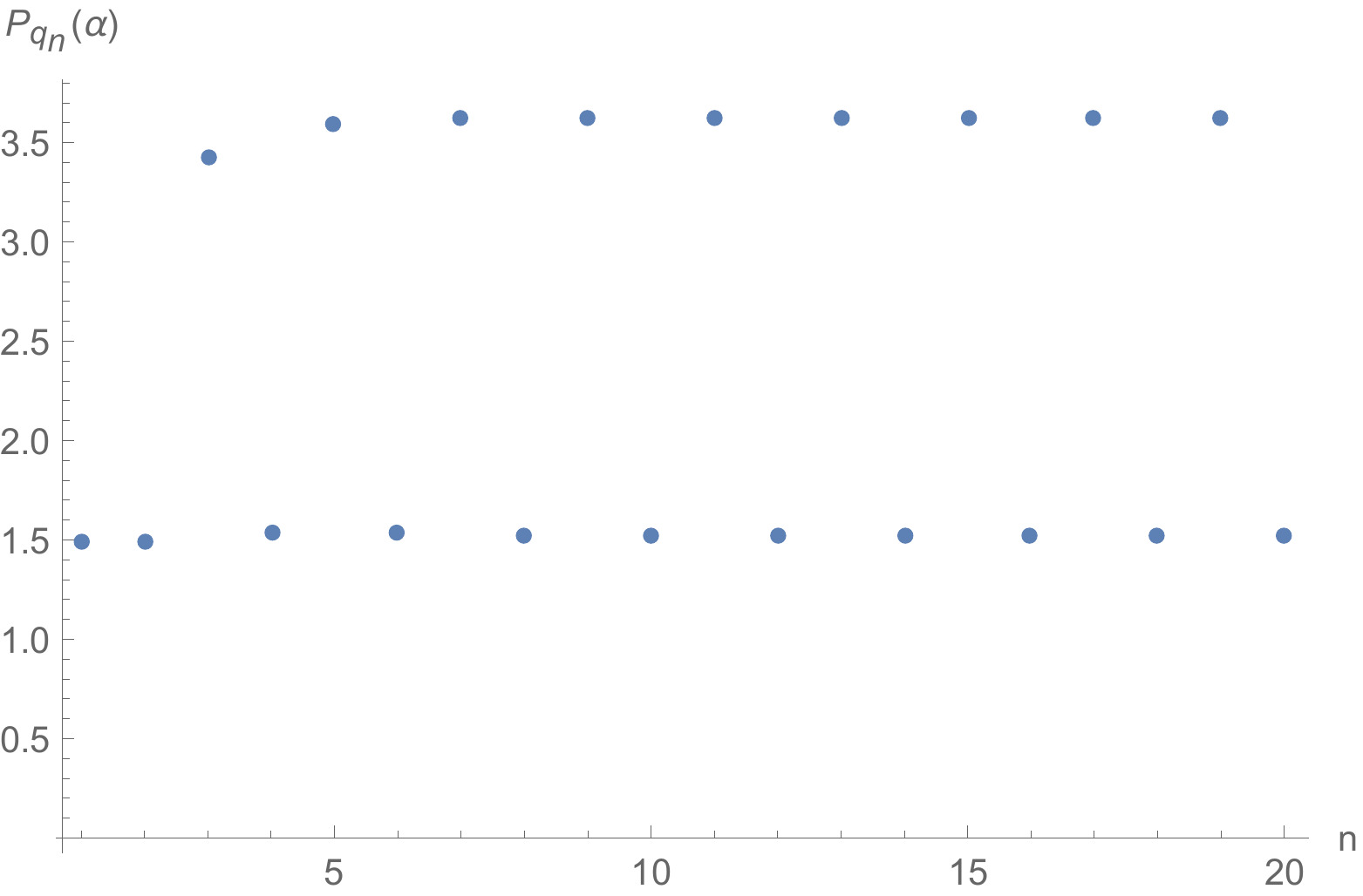}
	\caption{Values of $P_{q_n}(\alpha)$ for $\alpha=\varphi$ (above) and $\alpha=\sqrt{3}=[1;\overline{1,2}]$ (below), where $q_n$ is the $n$th best approximation denominator of $\alpha$.}
	\label{fig:convergenceex}
\end{figure}

\newpage

\section{A positive lower bound for $P_n(\alpha)$}\label{sec:liminf}
The limit phenomenon observed in Theorem \ref{thm:mvmain} sheds new light on the old and long-standing open problem of whether  
\begin{equation}
\label{eq:qliminf}
\liminf_{n\to \infty} P_n(\alpha) = 0 
\end{equation}
for all irrationals $\alpha$. As mentioned in Section \ref{subsec:cfc}, this question was raised by Lubinsky in \cite{Lu99}, but the problem goes back much further; also Erd\H{o}s and Szekeres asked this question already in the 1950s \cite{erdos}. Lubinsky showed that \eqref{eq:qliminf} indeed holds for all $\alpha$ with unbounded continued fraction coefficients, and suggested it is likely that \eqref{eq:qliminf} holds in general. 

However, when $\alpha=\varphi$ is the golden mean, numerics indicate that it is precisely along the subsequence $(F_n)_{n\geq 1}$ of Fibonacci numbers that $P_n(\varphi)$ takes on its minimum values. On the other hand, peaks of $P_n(\varphi)$ appear to be occurring along the subsequence $(F_n-1)_{n\geq1}$. Specifically, numerical calculations are suggesting that
\begin{equation}
\label{eq:possbound}
P_{F_{n-1}}(\varphi) \leq P_N(\varphi) \leq P_{F_n-1}(\varphi)
\end{equation}
for $n\geq 3$ and $N \in \{F_{n-1}, \ldots, F_n-1\}$. This is illustrated in Figure \ref{fig:minmax}.
\begin{figure}[htb]
	\centering
		\includegraphics[width=0.65\linewidth]{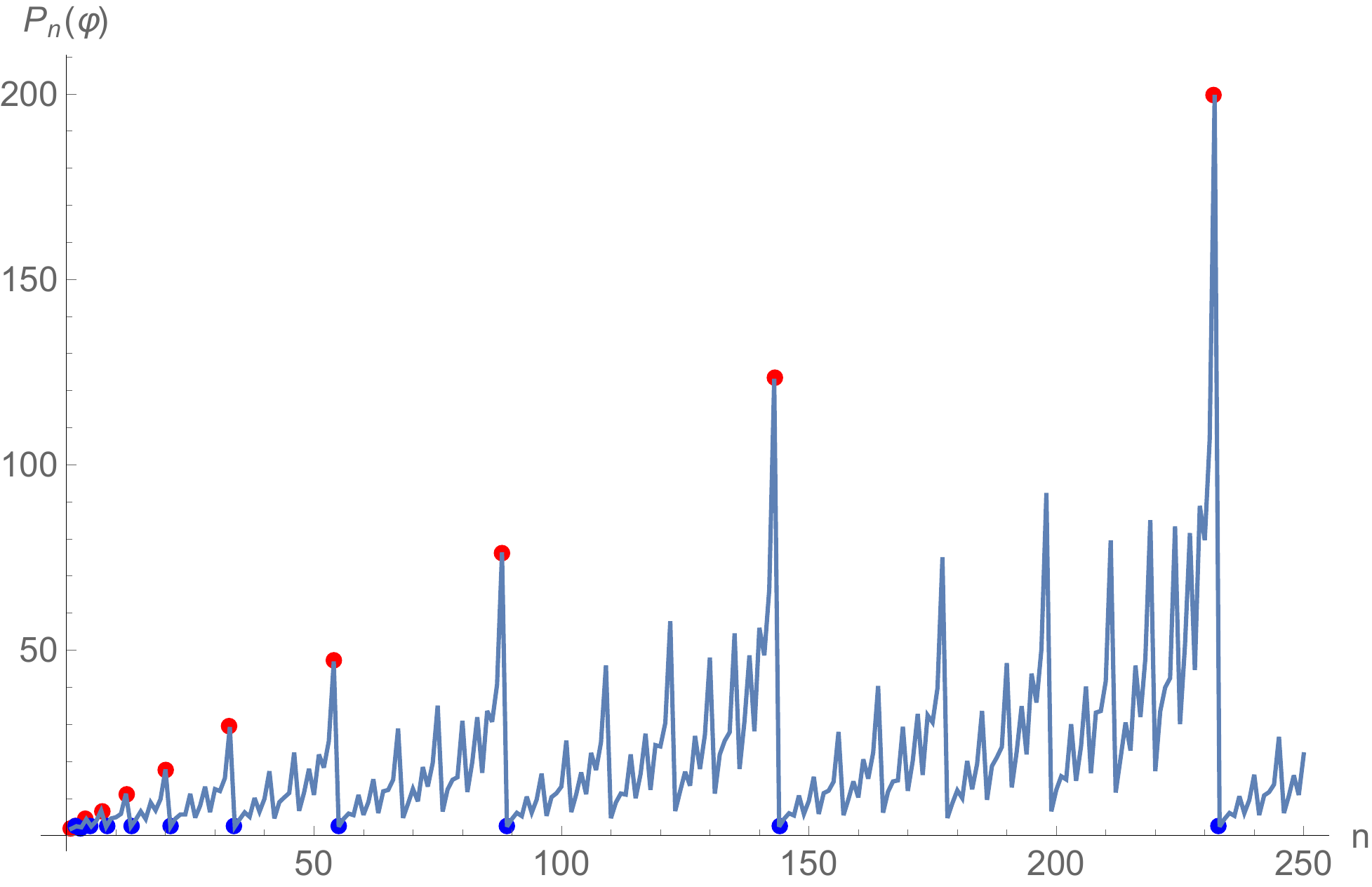}
	\caption{Value of $P_n(\varphi)$, with the two subsequences $P_{F_n}(\varphi)$ and $P_{F_n-1}(\varphi)$ indicated by blue and red marks, respectively.\label{fig:minmax}}
\end{figure}

The inequalities in \eqref{eq:possbound} have two immediate and important consequences. First of all, should the upper bound in \eqref{eq:possbound} hold, then it would follow that the growth of $P_n(\varphi)$ is at most linear. Using the convergence of the subsequence $P_{F_n}(\varphi)$, it is derived in \cite{MV16} that $P_{F_n-1}(\varphi) \leq cF_n$, and combining this with \eqref{eq:possbound} we get
\begin{equation*}
P_{N}(\varphi) \leq cF_n \leq 2cN.
\end{equation*}
Secondly, should the lower bound in \eqref{eq:possbound} hold, then it would follow immediately from Theorem \ref{thm:mvmain} that 
\begin{equation}
\label{eq:neqzero}
\liminf_{n \to \infty} P_n(\varphi) \geq \lim_{n\to\infty}P_{F_n}(\varphi)>0.
\end{equation}
To the best of our knowledge, the inequalities in \eqref{eq:possbound} have not been proven rigorously. Nevertheless, 
it turns out that \eqref{eq:neqzero} can be deduced from Theorem~\ref{thm:mvmain} by a slightly extended argument.
\begin{thm}[{\cite[Theorem 1.1]{GKN19}}]
\label{thm:liminfphi}
If $\varphi=(\sqrt{5}-1)/2$, then 
\begin{equation*}
\liminf_{n\to \infty} P_n(\varphi) = \liminf_{n\to \infty} \prod_{r=1}^n \left| 2 \sin \pi r \varphi \right| > 0. 
\end{equation*}
\end{thm}
The main idea in the proof of Theorem \ref{thm:liminfphi} is rather simple: 
For any $N\in \N$, let $N = \sum_{j=1}^{m} F_{n_j}$ be its Ostrowski representation in base $\varphi$ (also known as its Zeckendorf representation \cite{Z92}). We may then express $P_N(\varphi)$ as the double product
\begin{equation}
\label{eq:doubleprod}
P_N(\varphi) = \prod_{r=1}^N \left| 2\sin \pi r \varphi \right| = \prod_{j=1}^m \prod_{r=1}^{F_{n_j}} \left| 2\sin \pi(r\varphi+k_j\varphi)\right|,
\end{equation}
where $k_j = \sum_{s=j+1}^m F_{n_s}$ for $1\leq j\leq m-1$ and $k_m=0$. Observe that the inner product on the right hand side in \eqref{eq:doubleprod} is a perturbed version of $P_{F_j}(\varphi)$. It was shown in \cite[p.~220-221]{MV16} that for these perturbed products, there exist constants $0<K_1\leq 1 \leq K_2$ such that 
\begin{equation}
\label{eq:kbounds}
K_1 \leq \prod_{r=1}^{F_{n_j}} \left| 2\sin \pi(r\varphi+k_j\varphi)\right| \leq K_2 
\end{equation}
for all $1\leq j \leq m$. Now notice that the fractional part of the perturbation $k_j\varphi$ is tending to zero with increasing values of $j$. This is a consequence of the identity
\begin{equation*}
F_n\varphi = F_{n-1}-(-\varphi)^n.
\end{equation*}
We know from Theorem \ref{thm:mvmain} that the unperturbed sequence $P_{F_{n_j}}(\varphi)$ tends to a constant $c\approx 2.4$ as $j$ increases, and it is thus tempting to suggest that the lower bound $K_1\leq 1$ in \eqref{eq:kbounds} can be raised to some value greater than $1$ if $j$ is chosen sufficiently large. Indeed it turns out that  
\begin{equation*}
\prod_{r=1}^{F_{n_j}} \left| 2\sin \pi (r\varphi+k_j\varphi)\right| \geq 1, 
\end{equation*}
for all $j$ greater than some threshold value $J \in \N$ (independent of $N$), and accordingly it follows from \eqref{eq:doubleprod} and \eqref{eq:kbounds} that
\begin{equation*}
P_N(\varphi) \geq K_1^J
\end{equation*}
for all $N\in \N$.

For a detailed exposition of the proof of Theorem \ref{thm:liminfphi}, see \cite{GKN19}.


\section{Possible extensions of Theorem \ref{thm:liminfphi}}\label{sec:extensions}
We have now seen that $\liminf_{n\to \infty} P_n(\alpha) = 0$ fails for the golden mean $\alpha=\varphi$, and it is natural to ask whether  
\begin{equation*}
\liminf_{n\to \infty} P_n(\alpha) > 0 
\end{equation*}
also for other irrationals. Since the fact that $\liminf_{n\to \infty} P_n(\varphi)>0$ is deduced from Theorem \ref{thm:mvmain}, and Theorem \ref{thm:mvmain} has a natural extension to quadratic irrationals (Theorem \ref{thm:periodcase}), one is led to guess that Theorem \ref{thm:liminfphi} might be generalized to all quadratic irrationals. Unfortunately, this is too much to hope for. 
\begin{thm}
\label{thm:negresult}
Let $\alpha = [0;a_1, a_2, \ldots ]$ have bounded continued fraction coefficients, and let $M = \max_{j\in \N} a_j$. Provided $M$ is sufficiently large, there exists some threshold value $K=K(M)$ such that if $a_j \geq K$ infinitely often, then 
\begin{equation}
\label{eq:eqzero}
\liminf_{n \to \infty} P_n(\alpha) = 0.
\end{equation}
\end{thm} 
\begin{rem*}
Theorem \ref{thm:negresult} is a consequence of a result by Lubinsky (Proposition~\ref{prop:upperbound} below). Lubinsky himself claims in \cite{Lu99} that Theorem \ref{thm:negresult} is true for a general threshold $K$ independent of $M$. However, this is not rigorously proven, and we have not managed to verify it. Basing our argument on Proposition \ref{prop:upperbound} below, we do not see that the dependency on $M$ can be omitted.
\end{rem*}
\begin{prop}[{\cite[Proposition 5.1]{Lu99}}]
\label{prop:upperbound}
Let $\alpha =[0;a_1,a_2,\ldots]$, and for $n \in \N$ let $n=\sum_{j=1}^z b_jq_j$ be its Ostrowski expansion in base $\alpha$. Denote by $z^{\#}$ the length of this expansion
\begin{equation*}
z^{\#}=z^{\#}(n) = \# \left\{j \, : \, 1 \leq j \leq z \text{ and } b_j \neq 0 \right\}. 
\end{equation*}
We then have
\begin{equation}
\label{eq:upperbound}
\begin{aligned}
\log P_n(\alpha) &\leq 800z^{\#}+ 151 \sum_{j=1}^z \frac{b_j}{a_j} \max_{k<j} \log a_k+ \frac{3}{2} \sum_{j=1}^z\ \log^+ b_j \\
&+ \sum_{j=1}^z b_j \log \left( \frac{2\pi b_jq_j |q_j\alpha -p_j|}{e}\right),
\end{aligned}
\end{equation}
where $\log^+x = \max \{\log x, 0\}$.
\end{prop}
\begin{rem*}
The fact that $\liminf_{n\to\infty} P_n(\alpha)=0$ whenever $\alpha=[0;a_1, a_2, \ldots ]$ has \emph{unbounded} continued fraction coefficients is a straightforward consequence of this proposition (as illustrated by Lubinsky in \cite{Lu99}). To see this, simply construct a strictly increasing subsequence of coefficients $a_{n_j}$ where 
\begin{equation*}
a_{n_j} > a_k \quad \text{ for all } k < n_j .
\end{equation*}
Then putting $n=N_j = q_{n_j}$ in \eqref{eq:upperbound}, it is easily verified that this inequality reduces to 
\begin{equation*}
\log P_{N_j}(\alpha) \leq C - \log a_{n_j}
\end{equation*}
for some absolute constant $C$, and since $a_{n_j} \to \infty$ as $j\to \infty$ it follows that 
\begin{equation*}
\lim_{j \to \infty} P_{N_j}(\alpha) = 0.
\end{equation*}
\end{rem*}
Let us now see how Theorem \ref{thm:negresult} is deduced from Proposition \ref{prop:upperbound}.
\begin{proof}[Proof of Theorem \ref{thm:negresult}]
Let $\alpha = [0;a_1, a_2, \ldots]$ with $M =\max_j a_j$, and suppose $a_j \geq K$ infinitely often for some natural number $K \leq M$. Denote by $(n_i)_{i\in \N}$ a sequence of indices such that $a_{n_i} \geq K$ for every $i$. We may choose this sequence so that 
\begin{equation*}
n_i -n_{i-1} > 1 \quad \text{ for all } i>1 .
\end{equation*}

Now construct a sequence of integers $N_m$ by letting $N_m = \sum_{i=1}^m q_{n_i}$. We have then given $N_m$ in its Ostrowski representation to base $\alpha$, as
\begin{equation*}
N_m = \sum_{i=1}^m q_{n_i} = \sum_{j=1}^{n_m} b_j q_j,
\end{equation*}
where
\begin{equation*}
b_j = \begin{cases}
1, \quad &\text{ if } j \in (n_i)_{i\in \N} \\
0, \quad &\text{ otherwise}
\end{cases}
,
\end{equation*}
and where no two consecutive coefficients $b_j$ are both nonzero. 

We now use Proposition \ref{prop:upperbound} to estimate $\log P_{N_m} (\alpha)$. Since $b_j \in \{0,1\}$, it is clear that the third term on the right hand side in \eqref{eq:upperbound} is zero. For the second term on the right hand side in \eqref{eq:upperbound}, we have the upper bound 
\begin{equation}
\label{eq:secterm}
151 \sum_{j=1}^{n_m} \frac{b_j}{a_j} \max_{k<j} \log a_k \leq 151 \log M \sum_{i=1}^{m} \frac{1}{a_{n_i}} \leq \frac{151 \log M}{K} m.
\end{equation}
Finally, for the fourth term on the right hand side in \eqref{eq:upperbound}, we observe that if $b_j=1$, then 
\begin{equation*}
b_j \log \left( \frac{2\pi b_jq_j |q_j\alpha -p_j|}{e} \right) \leq \log \left( \frac{2\pi q_j}{eq_{j+1}}\right) \leq \log \left( \frac{\pi q_j}{a_jq_j}\right) = \log \pi - \log a_j,
\end{equation*}
where for the first inequality we have used \eqref{eq:approxerror}. It follows that
\begin{equation}
\label{eq:fourterm}
\sum_{j=1}^{n_m} b_j \log \left( \frac{2\pi b_jq_j |q_j\alpha -p_j|}{e} \right) \leq (\log \pi-\log K)m ,
\end{equation}
and inserting \eqref{eq:secterm} and \eqref{eq:fourterm} in \eqref{eq:upperbound}, we arrive at
\begin{equation}
\label{eq:posneg}
\log P_{N_m} (\alpha) \leq \left( 802 + \frac{151 \log M}{K} - \log K\right)m . 
\end{equation}
If $M$ is sufficiently small, then the right hand side in \eqref{eq:posneg} is positive regardless of the size of $K\leq M$. However, once $M$ is sufficiently large, one can find $K=K(M)$ such that
\begin{equation*}
802 + \frac{151 \log M}{K} - \log K < 0 .
\end{equation*}
In this case, it is clear from \eqref{eq:posneg} that 
\begin{equation*}
\log P_{N_m} (\alpha) \to -\infty
\end{equation*}
as $m\to \infty$, and accordingly
\begin{equation*}
\lim_{m \to \infty} P_{N_m}(\alpha)=0 .
\end{equation*}
This concludes the proof of Theorem \ref{thm:negresult}.
\end{proof}

\subsection{Irrationals of the form $\alpha=[0;\overline{a}]$}
Let us finally have an extra look at irrationals of the form 
\begin{equation*}
\alpha = [0; \overline{a}].
\end{equation*}
For this special case, we have $M=K=a$ in Theorem \ref{thm:negresult}, and it is clear from the proof that $\liminf_{n\to \infty} P_n(\alpha)=0$ if 
\begin{equation*}
802+\frac{151\log a}{a} - \log a < 0,
\end{equation*}
or equivalently if $a \geq e^{802+\varepsilon}$ for some small $\varepsilon>0$. 

Studying the product $P_n(\alpha)$ numerically, it appears that the true lower bound on $a$ for when $\liminf_{n\to\infty} P_n(\alpha) = 0$ might actually be significantly lower. 
In Table~\ref{tab:minima}, we have listed the evolution of minima of $P_n(\alpha)$ for $\alpha=[0; \overline{a}]$, $a=1,2, \ldots, 8$, determined numerically. 

\begin{table}[htb]
\centering
\begin{tabular}{c|c}
$\alpha$ & Evolution of minima $(P_n(\alpha), n)$\\
\hline
$[0;\overline{1}]$ & (1.865, 1) \\
$[0;\overline{2}]$ & (1.928, 1) \\
$[0;\overline{3}]$ & (1.333, 1) \\
$[0;\overline{4}]$ & (1.351, 1) \\
$[0;\overline{5}]$ & (1.138, 1) \\
$[0;\overline{6}]$ & (0.977, 1), (0.907, 7), (0.849, 44), (0.794, 272), (0.742, 1\,677), (0.693, 10\,335)\\
$[0;\overline{7}]$ & (0.852, 1), (0.708, 8), (0.589, 58), (0.491, 415), (0.408, 2\,964), (0.340, 21\,164)\\
$[0;\overline{8}]$ & (0.755, 1), (0.564, 9), (0.422, 74), (0.316, 602), (0.236, 4\,891), (0.177, 39\,731) \\
\end{tabular}
\caption{Evolution of minima of $P_n(\alpha)$ for $n=1, \ldots , 50\,000$. \label{tab:minima}}
\end{table}
It is curious that for $a\leq 5$, we have
\begin{equation*}
\min_{1\leq n \leq 50\,000} P_n(\alpha) = P_1(\alpha),
\end{equation*}
whereas for $a>5$, the minimal value of $P_n(\alpha)$ is decreasing slowly with increasing $n$. The apparent change in behaviour at the cutoff $a=5$ leads us to close by posing the following conjecture.
\begin{con}
Let $\alpha = [0;\overline{a}]$. If $a \leq 5$, then 
\begin{equation*}
\liminf_{n \to \infty} P_n(\alpha) \geq P_1(\alpha) > 0.
\end{equation*}
If $a>5$, then 
\begin{equation*}
\liminf_{n\to \infty} P_n(\alpha) = 0.
\end{equation*}
\end{con}


%

\Addresses

\end{document}